\documentclass[11pt,reqno]{amsart}

\usepackage{setspace}
\usepackage[all]{xy}
\usepackage{amssymb,amsmath}

\usepackage{enumitem}
 
\calclayout
\newtheorem{dummy}{anything}[section]
\newtheorem{theorem}[dummy]{Theorem}
\newtheorem{lemma}[dummy]{Lemma}
\newtheorem{proposition}[dummy]{Proposition}

\theoremstyle{definition}
\newtheorem{definition}[dummy]{Definition}
  \newtheorem{example}[dummy]{Example}
  
  \newtheorem{remark}[dummy]{Remark}

\newcommand{\bbZ}{\mathbb Z}
\newcommand{\bbQ}{\mathbb Q}
\newcommand{\bbR}{\mathbb R}

\newcommand{\bbF}{\mathbb F}

\newcommand{\cH}{\mathcal H}

\newcommand{\cO}{\mathcal O}
\newcommand{\cE}{\mathcal E}
\newcommand{\cM}{\mathcal M}
\newcommand{\cC}{\mathcal C}

\newcommand{\bC}{\mathbf C}
\newcommand{\bD}{\mathbf D}

\DeclareMathOperator{\Hom}{Hom}
\DeclareMathOperator{\End}{End}

\DeclareMathOperator{\Dim}{Dim}
\DeclareMathOperator{\Map}{Map} 
\DeclareMathOperator{\Or}{Or}

\DeclareMathOperator{\im}{im}

\DeclareMathOperator{\Res}{Res}
\DeclareMathOperator{\Def}{Def}

\DeclareMathOperator{\jn}{Join}
\DeclareMathOperator{\Jnd}{Jnd}
\DeclareMathOperator{\Ten}{Ten}
\DeclareMathOperator{\Aut}{Aut}
\newcommand{\bd}{\partial}
\newcommand{\id}{\mathrm{id}}

\newcommand{\un}{\underline}
\newcommand{\Sub}{\mathrm{Sub}}
\newcommand{\Flag}{\mathrm{Flag}}
\newcommand{\Hn}{\mathfrak{Hom}}

\def\G{\varGamma}

\def\maprt#1{\smash{\,\mathop{\longrightarrow}\limits^{#1}\,}}

\begin{document}

\title{Equivariant Moore spaces and the Dade group}
\author{Erg\" un Yal\c c\i n}

\address{Department of Mathematics, Bilkent
University, Ankara, 06800, Turkey.}

\email{yalcine@fen.bilkent.edu.tr}

\keywords{}

\thanks{2010 {\it Mathematics Subject Classification.} Primary: 57S17; Secondary: 20C20.}

\thanks{Partially supported by T\" UB\. ITAK-B\. IDEB-2219.}

\date{\today}

\begin{abstract} 
Let $G$ be a finite $p$-group and $k$ be a field of characteristic $p$. A topological space $X$ is called an $n$-Moore space if its reduced homology is nonzero only in dimension $n$. We call a $G$-CW-complex $X$ an $\underline{n}$-Moore $G$-space over $k$ if for every subgroup $H$ of $G$, the fixed point set $X^H$ is an $\underline{n}(H)$-Moore space with coefficients in $k$, where $\underline{n}(H)$ is a function of $H$.  We show that if $X$ is a finite $\underline{n}$-Moore $G$-space, then the reduced homology module of $X$ is an endo-permutation $kG$-module generated by relative syzygies. A $kG$-module $M$ is an endo-permutation module if ${\rm End}_k (M) =M \otimes _{k} M^*$ is a permutation $kG$-module. We consider the Grothendieck group of finite Moore $G$-spaces $\mathcal{M}(G)$, with addition given by the join operation, and relate this group to the Dade group generated by relative syzygies.    
\end{abstract}

\maketitle

\section{Introduction and statement of results} 

Let $G$ be a finite group and $M$ be a $\bbZ G$-module. A $G$-CW-complex $X$ is called a \emph{Moore $G$-space} of type $(M, n)$ if the reduced homology group $\widetilde H_i(X ; \bbZ)$ is zero whenever $i\neq n$ and $\widetilde H_n (X; \bbZ) \cong M$ as $\bbZ G$-modules. One of the classical problems in algebraic topology, due to Steenrod, asks whether every $\bbZ G$-module is realizable as the homology module of a Moore $G$-space. G. Carlsson \cite{Carlsson} constructed counterexamples of non-realizable modules for finite groups that include $\bbZ /p \times \bbZ/p$ as a subgroup for some prime $p$.  The question of finding a good algebraic characterization of realizable $\bbZ G$-modules is still an open problem (see \cite{JustinSmith} and \cite{Benson-Habegger}).

In this paper we consider Moore $G$-spaces whose fixed point subspaces are also Moore spaces. Let $R$ be a commutative ring of coefficients and let $\un{n} : \Sub(G) \to \bbZ$ denote a function from subgroups of $G$ to integers, which is constant on the conjugacy classes of subgroups. Such functions are often called super class functions.

\begin{definition}\label{def:main} A $G$-CW-complex $X$ is called an \emph{$\un{n}$-Moore $G$-space over $R$} if for every $H \leq G$, the reduced homology group $\widetilde H_i(X ^H ; R)$ is zero for all $i\neq \un{n}(H)$. 
\end{definition}

When $\un{n}$ is the constant function with value $n$ for all $H\leq G$, the homology at dimension $n$ can be considered as a module over the orbit category $\Or G$. If $\widetilde{\un{H}}_n (X^?; R) \cong \un{M}$ as a module over the orbit category, then $X$ is called a Moore $G$-space of type $(\un{M}, n)$. When $R=\bbQ$ and $X^H$ is simply-connected for all $H \leq G$, the space $X$ is called a rational Moore $G$-space. Rational Moore $G$-spaces are studied extensively in equivariant homotopy theory and many interesting results are obtained on homotopy types of these spaces (see \cite{Kahn} and \cite{Doman}). 

In this paper, we allow $\un{n}$ to be an arbitrary super class function and take the coefficient ring $R$ as a field $k$ of characteristic $p$. We define the group of finite Moore $G$-spaces over $k$ and relate this group to the Dade group, the group of endo-permutation modules. Since the appropriate definition of a Dade group for a finite group is not clear yet, we restrict ourselves to the situation where $G$ is a $p$-group, although the results have obvious consequences for finite groups via restriction to a Sylow $p$-subgroup.  
 
Let $G$ be a finite group and $k$ be a field of characteristic $p$. Throughout we assume all $kG$-modules are finitely generated. A $kG$-module $M$ is called an \emph{endo-permutation $kG$-module} if $\End _{k} (M)=M\otimes_k M^*$ is isomorphic to a permutation $kG$-module when regarded as a $kG$-module with diagonal $G$-action $(gf)(m)=gf(g^{-1} m)$. A $G$-CW-complex is called \emph{finite} if it has finitely many cells. The main result of the paper is the following:

\begin{theorem}\label{thm:main} Let $G$ be a finite $p$-group, and $k$ be a field of characteristic $p$. If $X$ is a finite $\un{n}$-Moore $G$-space over $k$, then the reduced homology module $\widetilde H_{n} (X, k)$ in dimension $n=\un{n}(1)$ is an endo-permutation $kG$-module generated by relative syzygies.
\end{theorem}
 
This theorem is proved in Sections \ref{sect:MooreGSpaces} and \ref{sect:ProofMain}. We first prove it for tight Moore $G$-spaces (Proposition \ref{pro:Tight}) and then extend it to the general case. An $\un{n}$-Moore space $X$ is said to be \emph{tight} if the topological dimension of $X^H$ is equal to $\un{n}(H)$ for every $H \leq G$. For tight Moore $G$-spaces, we give an explicit formula that expresses the equivalence class of the homology group $\widetilde H_n (X, k)$ in terms of relative syzygies (see Proposition \ref{pro:Tight}). This formula plays a key role for relating the group of Moore $G$-spaces to the group of Borel-Smith functions and to the Dade group. We now introduce these groups and the homomorphisms between them.  

An endo-permutation $G$-module is called \emph{capped} if it has an indecomposable summand with vertex $G$, or equivalently, if $\End _k (M)$ has the trivial module $k$ as a summand. There is a suitable equivalence relation of endo-permutation modules, and the equivalence  classes of capped endo-permutation modules form a group under the tensor product operation (see Section \ref{sect:DadeGroup}). This group is called the Dade group of the group $G$, denoted by $D_k(G)$, or simply by $D(G)$ when $k$ is clear from the context.

For a non-empty finite $G$-set $X$, the kernel of the augmentation map $\varepsilon \colon kX \to k$ is called a \emph{relative syzygy} and denoted by $\Delta(X)$. It is shown by Alperin \cite{Alperin} that $\Delta(X)$ is an endo-permutation $kG$-module and it is capped when $|X^G| \neq 1$. We define $\Omega _X \in D_k(G)$ as the element
$$\Omega _X= \begin{cases} [\Delta(X)] &\text{if $X\neq \emptyset$ \text{and} $|X^G|\neq 1$}; \\ 0   & \text{otherwise}.  \end{cases}$$
The subgroup of $D(G)$ generated by relative syzygies is denoted by $D^{\Omega} (G)$ and it plays an important role for understanding the Dade group.   

\begin{definition} We say a Moore $G$-space is \emph{capped} if $X^G$ has nonzero reduced homology. The set of $G$-homotopy classes of capped Moore $G$-spaces form a commutative monoid with addition given by $[X]+[Y]=[X*Y]$, where $X*Y$ denotes the join of two $G$-CW-complexes. We define the group of Moore $G$-spaces $\cM(G)$ as the Grothendieck group of this monoid.
\end{definition}

The \emph{dimension function} of an $\un{n}$-Moore $G$-space is defined as the super class function $\Dim X$ with values
$$(\Dim X) (H)=\un{n} (H)+1$$ for all $H \leq G$. Let $C(G)$ denote the group of all super class functions of $G$. The map $\Dim: \cM (G) \to C(G)$ which takes $[X]-[Y]$ to $\Dim X-\Dim Y$ is a group homomorphism since $\Dim (X* Y)=\Dim X + \Dim Y$. In Proposition \ref{pro:DimSurj}, we prove that the homomorphism $\Dim$ is surjective. This follows from the fact that $C(G)$ is generated by super class functions of the form $\omega _X$, where $X$ denotes a finite $G$-set and $\omega _X$ is the function defined by $$ \omega _X (K) =\begin{cases} 1 & \text{if $X^K \neq \emptyset$} \\ 0 & \text{otherwise} 
\end{cases} $$ for all $K \leq G$. Note that if we consider a finite $G$-set $X$ as a discrete $G$-CW-complex, then $X$ is a finite Moore $G$-space with dimension function $\Dim X=\omega _X$.

We also define a group homomorphism $\Hn: \cM (G) \to D^{\Omega }(G)$ as a linear extension of the assignment that takes the equivalence class $[X]$ of a capped $\un{n}$-Moore $G$-space to the equivalence class of its reduced homology $[\widetilde H_n (X; k)]$ in $D^{\Omega} (G)$, where $n=\un{n} (1)$.  There is also a group homomorphism $\Psi : C(G) \to D^{\Omega} (G)$ that takes $\omega_X$ to $\Omega _X$ for every $G$-set $X$ (see \cite[Theorem 1.7]{Bouc-Aremark}). In Proposition \ref{pro:DependsOnly}, we show that $$\Hn =\Psi \circ \Dim.$$ This gives in particular that for an $\un{n}$-Moore $G$-space $X$, the equivalence class of its reduced homology $[\widetilde H_n (X; k)]$ in $D^{\Omega}(G)$ is uniquely determined by the function $\un{n}$. Moreover we prove the following theorem. 

\begin{theorem}\label{thm:mainSeq} Let $G$ be a finite $p$-group and $k$ a field of characteristic $p$.  Let $\cM _0 (G)$ denote the kernel of the homomorphism $\Hn$, and $C_b(G)$ denote the group of Borel-Smith functions (see \cite[Definition 3.1]{Bouc-Yalcin}). Then, there is a commuting diagram
$$\xymatrix{
0 \ar[r] & \cM_0(G) \ar[r] \ar[d]^{\Dim_0} &  \cM(G) \ar[d]^{\Dim} \ar[r]^{\Hn}
& D^{\Omega} (G)  \ar[d]^{=}  \ar[r] & 0  \\
0  \ar[r] & C_b(G) \ar[r] &  C(G) \ar[r]^{\Psi}
&  D^{\Omega} (G) \ar[r] & 0}\\ $$
where the maps $\Dim$ and $\Dim _0$ are surjective and the horizontal sequences are exact.
\end{theorem} 

In the proof of the above theorem we do not assume the exactness of the bottom sequence. It follows from the exactness of the 
top sequence and from the fact that the maps $\Dim$ and $\Dim _0$ are surjective (surjectivity of $\Dim _0$ follows from a theorem of Dotzel-Hamrick \cite{Dotzel-Hamrick}).  Note that the exactness of the bottom sequence is the main result of Bouc-Yal{\c c}{\i}n \cite{Bouc-Yalcin} and the proof given there is completely algebraic. The proof we obtain here can be considered as a topological interpretation of this short exact sequence.

In Section \ref{sect:Operations}, we consider operations on Moore $G$-spaces induced by actions of bisets on Moore $G$-spaces. 
We show that the assignment $G \to \cM(G)$ over all $p$-groups has an easy to describe biset functor structure, where the induction is given by join induction, and that $\Hn$ and $\Dim$ are natural transformations of biset functors. The induction operation on $\cM(-)$ is defined using join induction of $G$-posets $\jn _K ^H X$ and the key result here is that the homology of a join induction $\jn _K ^H X$ is isomorphic to the tensor induction of the homology of $X$. 
Using this we obtain a topological proof for Bouc's tensor induction formula for relative syzygies (see Theorem \ref{thm:TensorInd}). As a consequence we conclude that the diagram in Theorem \ref{thm:mainSeq} is a diagram of biset functors.  

The paper is organized as follows: In Section \ref{sect:DadeGroup}, we introduce some necessary definitions and background on Dade groups. We prove Theorem \ref{thm:main}
in Sections \ref{sect:MooreGSpaces} and \ref{sect:ProofMain}. In Section \ref{sect:GroupGMoore}, we introduce the group of Moore $G$-spaces $\cM(G)$ and prove Theorem \ref{thm:mainSeq}. In Section \ref{sect:Operations}, we define a biset functor structure for $\cM(-)$ and show that the diagram in Theorem \ref{thm:mainSeq} is a diagram of biset functors.

\section{Preliminaries on the Dade group}
\label{sect:DadeGroup}

Let $p$ be a prime number, $G$ be a finite $p$-group, and $k$ be a field of characteristic $p$. Throughout we assume that all $kG$-modules are finitely generated. A (left) $kG$-module $M$ is called an \emph{endo-permutation module} if $\End _{k} (M) \cong M\otimes_k M^*$ is isomorphic to a permutation $kG$-module. Here we view $\End _k (M)$ as a $kG$-module with diagonal action given by $(g\cdot f)(m)=gf(g^{-1}m)$. In this section we introduce some basic definitions and results on endo-permutation $kG$-modules that we will use in the paper. For more details on this material, we refer the reader to \cite[sect. 12.2]{Bouc-BisetBook} or \cite{Bouc-Tensor}.

Two endo-permutation $kG$-modules $M$ and $N$ are said to be \emph{compatible} if $M\oplus N$ is an endo-permutation $kG$-module. This is equivalent to the condition that $M \otimes _k N^*$ is a permutation $kG$-module (see \cite[Definition 12.2.4]{Bouc-BisetBook}). When $M$ and $N$ are compatible, we write $M\sim N$. An endo-permutation module $M$ is called \emph{capped} if it has an indecomposable summand with vertex $G$, or equivalently, if $\End_k(M)$ has the trivial module $k$ as a summand (see \cite[Lemma 12.2.6]{Bouc-BisetBook}). 
The relation $M \sim N$ defines an equivalence relation on capped endo-permutation $kG$-modules (see \cite[Theorem 12.2.8]{Bouc-BisetBook}).

Every capped endo-permutation module $M$ has a capped indecomposable summand, called the cap of $M$. Note that if $V$ is a cap of $M$, then $V\otimes_k M^*$ is a summand of $M\otimes_k M^*$ which is a permutation $kG$-module. This gives that $V\otimes_k M^*$ is a permutation $kG$-module, hence $V\sim M$. If $W$ is another capped indecomposable summand of $M$, then $V \cong W$ (see \cite[Lemma 12.2.9]{Bouc-BisetBook}), so the cap of $M$ is unique up to isomorphism. Two capped endo-permutation $kG$-modules are equivalent if and only if they have isomorphic caps (see \cite[Remark 12.2.11]{Bouc-BisetBook}).

The set of equivalence classes of capped endo-permutation modules has an abelian group structure under the addition given by $[M]+[N]=[M\otimes _k N]$. It is easy to see that this operation is well-defined (see \cite[Theorem 12.2.8]{Bouc-BisetBook}). This group is called the \emph{Dade group} of $G$ over $k$ and is denoted by $D_k(G)$, or simply by $D(G)$ when the field $k$ is clear from the context.  

For a non-empty $G$-set $X$, the kernel of the augmentation map $\varepsilon : kX \to k$ is called a \emph{relative syzygy} and is denoted by $\Delta(X)$. It is shown by Alperin \cite[Theorem 1]{Alperin} that $\Delta(X)$ is an endo-permutation module and it is capped if and only if $|X^G| \neq 1$ (see also \cite[Section 3.2]{Bouc-Tensor}). For a $G$-set $X$, let $\Omega _X$ denote the element in the Dade group $D(G)$ defined by $$\Omega _X= \begin{cases} [\Delta(X)] &\text{if $X\neq \emptyset$ \text{and} $|X^G|\neq 1$}; \\ 0   & \text{otherwise}.  \end{cases}$$  Note that if $X^G \neq \emptyset$, the module $\Delta (X)$ is a permutation module, so in this case we have $\Omega _X=0$. The subgroup of the Dade group generated by the set of elements $\Omega _X$, over all finite $G$-sets $X$, is denoted $D^{\Omega} (G)$ and called the Dade group generated by relative syzygies.

\begin{remark}\label{rem:coefficients} Let $\cO$ denote a complete noetherian local ring with residue field $k$ of characteristic $p>0$. The notion of an endo-permutation module and the Dade group can be extended to $\cO G$-modules which are $\cO$-free (called $\cO G$-lattices). 
In this case the Dade group is denoted by $D_{\cO}(G)$ and there is a natural map $\varphi : D_{\cO}(G) \to D_k (G)$ defined by reduction of coefficients. An element $x \in D_k (G)$ is said to have an integral lift if $x=\varphi (\overline x)$ for some $\overline x \in D_{\cO}(G)$. By definition, the elements of $D^{\Omega} (G)$ have integral lifts. This means that when we are working with $D^{\Omega}(G)$, it does not matter if we take the coefficients as $\cO$ or $k$. Note also that a relative syzygy over $k$ is obtained from an endo-permutation $\bbF _pG$-module via tensoring with $k$ over $\bbF_p$. In particular, the group $D^{\Omega} (G)$ does not depend on the field $k$ as long as it is a field with characteristic $p$.
\end{remark}

Now we are going to state some results related to relative syzygies that we are going to use later in the paper. In \cite[Section 3.2]{Bouc-Tensor} these results are stated in $\cO$-coefficients, but they also hold in $k$-coefficients. So in the results stated below $R$ is a commutative coefficient ring which is either a field $k$ of characteristic $p$, or a complete noetherian local ring $\cO$ with residue field $k$ of characteristic $p$. We refer the reader to \cite[Section 3.2]{Bouc-Tensor} for more details.  

\begin{definition} Let $G$ be a finite group and $X$ be a finite $G$-set. A sequence of $RG$-modules $0 \to M_0 \to M_1 \to M_1 \to 0$ is called \emph{$X$-split} if the corresponding sequence $$ 0 \to RX \otimes _R M_0 \to RX \otimes _R M_1 \to RX \otimes _R M_2 \to 0,$$ obtained by tensoring everything with $RX$, splits. 
\end{definition}

There is an alternative criterion for a sequence to be $X$-split. 

\begin{lemma}\label{lem:XSplit} Let $G$ be a finite group and $X$ be a finite $G$-set. A sequence of $RG$-modules is $X$-split if and only if it splits as a sequence of $RG_x$-modules for every stabilizer $G_x$ in $G$. 
\end{lemma}

\begin{proof} See \cite[Lemma 2.6]{Nucinkis}.
\end{proof}

We now state the main technical result that we will use in the paper.

\begin{lemma}\label{lem:HellerShift}
Let $G$ be a $p$-group and $X$ be a finite non-empty $G$-set. Suppose that $$0\to W \to RX \to V \to 0 $$ is an $X$-split exact sequence of $RG$-lattices. Then,
\begin{enumerate}
\item The lattice $V$ is an endo-permutation $RG$-lattice if and only if $W$ is an endo-permutation $RG$-lattice.
\item If $X^G = \emptyset$, then $V$ is capped if and only if $W$ is capped.
\item If $V$ and $W$ are capped endo-permutation $RG$-lattices, then $$W=\Omega _X +V$$ in $D_R(G)$. 
\end{enumerate}  
\end{lemma}

\begin{proof} See \cite[Lemma 3.2.8]{Bouc-Tensor}.
\end{proof}

The following also holds:

\begin{lemma}\label{lem:Secondlemma}
Let $G$ be a $p$-group. Suppose that $X$ and $Y$ are two non-empty finite $G$-sets such that for any subgroup $H$ of $G$, the set $X^H$ is non-empty if and only if $Y^H$ is non-empty. Then $\Omega _X=\Omega _Y$ in $D_R(G)$.
\end{lemma} 

\begin{proof} See \cite[Lemma 3.2.7]{Bouc-Tensor}.
\end{proof}

\section{Algebraic Moore $G$-complexes}
\label{sect:MooreGSpaces}

Let $G$ be a finite group and $\cH$ be a family of subgroups of $G$ closed under conjugation and taking subgroups. The orbit category $\Or _{\cH} G$ over the family $\cH$ is defined as the category whose objects are transitive $G$-sets $G/H$ where $H \in \cH$, and whose morphisms are $G$-maps $G/H \to G/K$. Throughout this paper we assume that the family $\cH$ is the family of all subgroups of $G$ and we denote the orbit category simply by $\G_G :=\Or G$. 

Let $R$ be a commutative ring of coefficients. An $R\G _G$-module $M$ is a contravariant functor from the category $\G_G$ to the category of $R$-modules. The value of an $R\G_G$-module $M$ at $G/H$ is denoted $M(H)$. By identifying $\Aut _{\G_G } (G/H)$ with $W_GH :=N_G(H)/H$, we can consider $M(H)$ as a $W_G(H)$-module. In particular, $M(1)$ is an $RG$-module. 

The category of $R\G_G$-modules is an abelian category, so the usual concepts of projective module, exact sequence, and chain complexes are available. 
For more information on modules over the orbit category, we refer the reader to L\"uck \cite[\S 9, \S 17]{Lueck} and tom Dieck \cite[\S 10-11]{tDieck-BlueBook}.

\begin{definition}\label{defn:Free} For a $G$-set $X$, we define $R\G_G$-module $R[X^?]$ as the module with values at $G/H$ given by $R[X^H]$, with obvious induced maps. A module over the orbit category $\Or_{\cH} G$ is called \emph{free} if it is isomorphic to a direct sum of modules of the form $R[(G/K)^?]$ with $K \in \cH$. By the Yoneda lemma every free $R\G_G$-module is projective (see \cite[Section 2A]{HPY}). 
\end{definition}

Let $X$ be a $G$-CW-complex. The reduced chain complex of $X$ over the orbit category is the functor $\widetilde C_* (X^?; R)$ from orbit category $\G_G$ to the category of chain complexes of $R$-modules, where for each $H\leq G$, the object $G/H$ is mapped to the reduced cellular chain complex  $\widetilde C_* (X^H ; R)$. This gives rise to a chain complex of $R\G_G$-modules 
$$ \widetilde C_* (X^?, R) : \cdots \to C_i (X^? ; R) \maprt{\bd _i} C_{i-1} (X^? ; R) \to \dots \to C_0 (X^? ; R) \maprt{\varepsilon} \un{R} \to 0$$
with boundary maps given by $R\G_G$-module maps between the chain modules $C_i (X^?; R)$, where for each $i\geq 0$, the chain module $C_i (X^?; R)$ is the $R\G_G$-module defined by $G/H \to C_i (X^H ; R)$. 

In the above sequence $\un{R}$ denotes the constant functor with values $\un{R} (H)=R$ for each $H \leq G$ and the identity map $\id : R \to R$ as the induced map $f^* \colon \un{R}(G/H) \to \un{R} (G/K)$ between $R$-modules for every $G$-map $f: G/K \to G/H$. The augmentation map $\varepsilon$ is defined as the $R\G_G$-homomorphism such that for each $H \leq G$, the map $\varepsilon (H): C_0(X^H; R) \to R$ is the $R$-linear map which takes every $0$-cell $\sigma \in X^H$ to $1$. By convention we take $\widetilde C_{-1} (X^?; R)=\un{R}$ and $\bd _{0}=\varepsilon$.

\begin{lemma}\label{lem:projective} The reduced chain complex $\widetilde C_*(X^?; R)$ of a $G$-CW-complex $X$ is a chain complex of free $R\G_G$-modules.
\end{lemma}

\begin{proof} If $i\geq 0$, then for each $H \leq G$, the chain module $C_i (X^H ; R)$ is isomorphic to the permutation module $R[X_i ^H]$, where $X_i $ is the $G$-set of $i$-dimensional cells in $X$. This gives an isomorphism of $R\G_G$-modules $C_i (X^?; R) \cong R[X_i ^? ]
$, hence $C_i (X^?; R)$ is a free $R\G _G$-module for every $i \geq 0$ (see \cite[9.16]{Lueck} or \cite[Ex. 2.4]{HPY}). The constant functor $\un{R}$ is isomorphic to the module $R[(G/G) ^?]$ which is a free $R\G_G$-module because we assumed $\cH$ is the family of all subgroups of $G$, in particular, $G \in \cH$. Hence  $C_i (X^?; R)$ is free for all $i\geq -1$.
\end{proof}

In the rest of the section we state our results for chain complexes of projective modules over the orbit category. We assume that all the chain complexes we consider are \emph{bounded from below}, i.e., there is an integer $s$ such that $\bC_i=0$ for all $i<s$.
We say $\bC$ is \emph{finite-dimensional} if there is an $n$ such that $\bC_i = 0$ for all $i\geq n+1$. If $\bC\neq 0$, then the smallest such integer is called the dimension of $\bC$.  For more information on chain complexes over the orbit category, we refer the reader to \cite[\S 2]{Hambleton-Yalcin-HomRep} or \cite[\S 2, \S6]{HPY}.

\begin{definition}\label{def:MooreComplex} Let $\bC$ be a chain complex of projective $R\G_G$-modules and let $\un{n}: \Sub(G) \to \bbZ$ be a super class function. We call $\bC$ an \emph{$\un{n}$-Moore $R\G_G$-complex} if for every $H \in \cH$, the homology group $H_i(\bC(H))$ is zero for every $i\neq \un{n}(H)$. We say $\bC$ is \emph{tight} if  for every $H \in \cH$, the chain complex $\bC (H)$ is non-zero and has dimension equal to $\un{n}(H)$. A Moore $R\G_G$-complex $\bC$ is called \emph{capped} if $G \in \cH$ and $H_* (\bC(G))$ is non-zero.
\end{definition}

If $X$ is an $\un{n}$-Moore $G$-space over $R$ as in Definition \ref{def:main}, then by Lemma \ref{lem:projective}, the reduced chain complex $\widetilde C_* (X^?; R)$ is an $\un{n}$-Moore $R\G_G$-complex. Moreover, if $X$ is a capped Moore space, then the chain complex $\widetilde C_* (X^?; k)$ is a capped $R\G_G$-complex.  

\begin{lemma}\label{lem:main} Let $\bC$ be a chain complex of projective $R\G_G$-modules (bounded from below). Suppose that $\bC$ is a tight $\un{n}$-Moore $R\G_G$-complex and $H$ is a subgroup of $G$. Then, for every $i\leq \un{n}(H)$, the short exact sequence $$ 0 \to \ker \bd_i \to \bC _i \to \mathrm{im}\, \bd _i \to 0$$ splits as a sequence of $R\G_H$-modules, where $\G _H=\Or H$.
\end{lemma} 

\begin{proof} Let $s$ be an integer such that $\bC_i=0$ for $i<s$. For every $K\leq H$, the chain complex $\bC(K)$ has zero homology except in dimension $\un{n}(K)$, which is equal to the chain complex dimension of $\bC(K)$. The dimension function of a projective chain complex is monotone (see \cite[Definition 2.5, Lemma 2.6]{Hambleton-Yalcin-HomRep}). Hence we have $\un{n}(K) \geq \un{n} (H)$ for every $K\leq H$. This gives that the truncated complex
\begin{equation}\label{Eqn:Sequence} 0 \to \ker \bd _{\un{n} (H)} \to \bC_{\un{n} (H) } \to \cdots \to \bC_s \to 0
\end{equation}
is exact when it is considered as a sequence of modules over $R\G _H$. Note that $\G_H$ is a subcategory of $\G_G$, so there is an induced restriction map $\Res ^G _H$ that takes projective $R\G_G$-modules to projective $R\G_H$-modules (see \cite[Proposition 3.7]{HPY}). This implies that $\Res ^G _H \bC_i$ is a projective $R\G_H$-module for every $i\geq s$, hence the sequence (\ref{Eqn:Sequence}) splits as a sequence of $R\G_H$-modules.  
\end{proof}

Note that for the algebraic theory of Moore $G$-spaces, it is enough to consider projective $R\G_G$-modules, but for obtaining results related to endo-permutation modules one would need these chain complexes to be chain complexes of free $R\G_G$-modules. When $G$ is a finite $p$-group and $k$ is a field of characteristic $p$, these two conditions are equivalent.

\begin{lemma}\label{lem:projfree} Let $G$ be a finite $p$-group and $k$ be a field of characteristic $p$. Then every projective $k\G_G$-module is free.
\end{lemma}

\begin{proof} By \cite[Corollary 9.40]{Lueck}, every projective $k\G_G$-module $P$ is a direct sum of modules of the form $E_H S_H P$ where $H\leq G$ and $E_H$ and $S_H$ are functors defined in \cite[pg. 170]{Lueck}. Since the functor $S_H$ takes projectives to projectives, $S_H P $ is a projective $kN_G(H)/H$-module (see \cite[Lemma 9.31(c)]{Lueck}). The group $N_G(H)/H$ is a $p$-group and $k$ is a field of characteristic $p$, hence $S_H P$ is a free module. The functor $E_H $ takes free modules to free $k\G _G$-modules (see \cite[Lemma 9.31(c)]{Lueck}), therefore $P$ is a free $k\G_G$-module.
\end{proof}

Now we are ready to prove an algebraic version of Theorem \ref{thm:main} for tight complexes.
Recall that a chain complex $\bC$ over $R\G_G$ is \emph{finite} if it is bounded and has the property that for each $i$, the chain module $\bC_i$ is finitely-generated as an $R\G _G$-module.  

\begin{proposition}\label{pro:main} Let $G$ be a finite $p$-group and $k$ be a field of characteristic $p$. If $\bC$ is a finite tight Moore $k\G_G$-complex, then $H_n (\bC(1))$ is an endo-permutation $kG$-module, where $n=\un{n}(1)$.
\end{proposition}

\begin{proof} By Lemma \ref{lem:projfree}, we can assume that $\bC$ is a finite chain complex of free $k\G_G$-modules with boundary maps $\bd_i : \bC_i \to \bC_{i-1}$. For each $i$, let $X_i$ denote the $G$-set such that $\bC_i \cong k[X_i ^?]$ as a $k\G_G$-module. If we evaluate $\bC$ at $1$ and augment the complex with the homology module, then we obtain an exact sequence of $RG$-modules 
$$ 0 \to H_n (\bC(1)) \to k[X_n]\to \cdots \to k[X_i] \maprt{\bd_i'} k[X_{i-1}] \to \cdots \to k[X_s] \to 0$$ where $\bd _i '=\bd _i (1)$.
To show that $H_n (\bC (1))$ is an endo-permutation module, we will inductively apply Lemma \ref{lem:HellerShift}(i)  to each of the extensions of $kG$-modules
$$ \cE _i : 0 \to \ker \bd' _i \to k[X_i] \to \im \bd' _i \to 0$$ for all $i$ such that $s \leq i \leq n$. 
If $H=G_x$ for some $x\in X_i$, then we have $i\leq \un{n} (H)$ because $\un{n}(H)$ is equal to the dimension of the chain complex $\bC(H)$ by the tightness condition. By Lemma \ref{lem:main}, the sequence
$$ 0 \to \ker \bd_i \to \bC_i \to \mathrm{im}\, \bd _i \to 0$$
splits as a sequence of $k\G _H$-modules, hence the sequence $\cE _i$  splits as a sequence of $kH$-modules. Since this is true for the isotropy subgroups $G_x$ of all the elements $x\in X_i$, by Lemma \ref{lem:XSplit} we conclude that the sequence $\cE _i$ is $X_i$-split. Hence by applying Lemma \ref{lem:HellerShift}(i) inductively, we conclude that $H_n(\bC (1))$ is an endo-permutation $kG$-module.
\end{proof}

We now give a more explicit formula for the equivalence class of the reduced homology 
$H_n (\bC (1))$ in the Dade group.
 
\begin{proposition}\label{pro:computing} Let $G$ be a finite $p$-group and $k$ be a field of characteristic $p$. Let $\bC$ be a finite tight $\un{n}$-Moore $k\G_G$-complex such that $\bC _i \cong k[X_i ^?]$ for each $i$. If $\bC$ is capped, then $H_{n} (\bC(1))$ is a capped endo-permutation $kG$-module and the formula $$[H_n (\bC(1))]=\sum _{i=m+1} ^{n} \Omega _{X_i}$$ holds in $D^{\Omega}(G)$, where $n=\un{n} (1)$ and $m=\un{n}(G)$.
\end{proposition}

\begin{proof}  As before we have an exact sequence of $kG$-modules  
$$ 0 \to H_n (\bC(1)) \to k[X_n]\maprt{\bd' _n} \cdots \to k[X_m] \maprt{\bd' _m} \cdots \to k[X_s] \to 0$$
where $m=\un{n}(G) \leq n=\un{n}(1)$. We claim that $Z_m= \ker \bd' _m $ is a capped permutation $kG$-module, i.e., $Z_m$ is a permutation $kG$-module that includes the trivial module $k$ as a summand. Once the claim is proved, by Lemma \ref{lem:HellerShift} (ii) and (iii), we can conclude that $H_n (\bC(1))$ is a capped endo-permutation module and the formula given above holds. 

We will show that $Z_m$ is a permutation $kG$-module such that the trivial module $k$ is one of the summands. To show this, first note that by Lemma \ref{lem:main}, the sequence $$0 \to \ker \bd_m \to \bC_m \maprt{\bd_m} \cdots \to \bC _s \to 0 $$ is a split exact sequence
of $k\G _G$-modules. This gives that $\ker \bd_m$ is a projective $k\G_G$-module. By Lemma \ref{lem:projfree}, every projective $k\G_G$-module is a free module. Hence  $\ker \bd _m \cong \oplus_i k[G/H_i ^?]$ for some subgroups $H_i \leq G$. From this we obtain that $Z_m \cong \oplus_i k[G/H_i]$, hence $Z_m$ is a permutation $kG$-module.

Note that the summands of $\bC_i$ that are of the form $k[(G/G)^?]$ form a subcomplex of $\bC$, denoted by $\bC ^G$, and we have $\bC(G)=\bC^G (G)$. Since $\bC$ is a capped Moore $R\G_G$-complex, the complex $\bC(G)$ has nontrivial homology at dimension $m$. This implies that the subcomplex $$0 \to \bC_m ^G \maprt{\bd_m^G}  \cdots \to \bC_s ^G \to 0$$ also has nontrivial homology at dimension $m$. This gives that $\ker \bd _m $ has a nontrivial summand of the form $k[(G/G)^?]$. Hence $Z_m$ includes the trivial module $k$ as a summand. 
\end{proof}

As a corollary of the results of this section, we obtain the following result.

\begin{proposition}\label{pro:Tight} Let $G$ be a finite $p$-group and $k$ be a field of characteristic $p$. Let $X$ be a finite tight $\un{n}$-Moore $G$-space over $k$, and let $X_i$ denote the $G$-set of $i$-dimensional cells in $X$ for each $i$. Then the reduced homology group $\widetilde H_{n} (X, k)$, where $n=\un{n}(1)$, is an endo-permutation $kG$-module. Moreover if $X$ is also capped, then $\widetilde H_n (X, k)$ is a capped endo-permutation $kG$-module, and the formula $$[\widetilde H_{n} (X; k)]=\sum _{i=m+1}^n \Omega _{X_i}$$ holds in the Dade group $D^{\Omega} (G)$, where $m=\un{n} (G)$.
\end{proposition}

\begin{proof} By Lemma \ref{lem:projective} the reduced chain complex $\bC:=\widetilde C_* (X^?; k)$ is a finite complex of free $k\G _G$-modules. Applying Propositions \ref{pro:main} and \ref{pro:computing} to the chain complex $\bC$, we obtain the desired conclusions.
\end{proof}

\section{Proof of Theorem \ref{thm:main}}
\label{sect:ProofMain}
 
In Section \ref{sect:MooreGSpaces} we proved that the conclusion of Theorem \ref{thm:main} holds for tight Moore $G$-spaces. To extend this result to an arbitrary finite Moore $G$-space, we show that up to taking joins with representation spheres, all Moore $G$-spaces have tight chain complexes up to chain homotopy. We first start with a brief discussion of the join operation on Moore $G$-spaces.

Let $G$ be a discrete group. Given two $G$-CW-complexes $X$ and $Y$, the join $X\ast Y$ is defined as the quotient space $X\times Y \times [0,1] /\sim $ with the identifications $(x,y,1)\sim (x',y,1)$ and $(x,y,0)\sim (x, y', 0)$ for all $x,x'\in X$ and $y,y'\in Y$. The $G$-action on $X\ast Y$ is given by $g(x,y,t)=(gx, gy, t)$ for all $x\in X$, $y\in Y$, and $t\in [0,1]$. To avoid the usual problems in algebraic topology with products of topological spaces, we assume the topology on products of spaces is the compactly generated topology. Then the join $X*Y$ has a natural a $G$-CW-complex structure. 

The $G$-CW-complex structure on $X\ast Y$ can be taken as the $G$-CW-structure inherited from the union $(CX \times Y) \cup _{X\times Y} (X\times CY)$ where $CX$ and $CY$ denote the cones of $X$ and $Y$, respectively. The CW-complex structure on the products $CX\times Y$ and $X\times CY$ are the usual $G$-CW-complex structures for products that we explain below. 

Given two $G$-CW-complexes $X$ and $Y$, the $G$-CW-structure on $X\times Y$ can be described as follows: Given two orbits of cells $G/H \times e^p$ and $G/K \times e^q$ in $X$ and $Y$, with attaching maps $\varphi$ and $\psi$, in the product complex we have a disjoint union of orbits of cells $$\coprod _{HgK\in H \backslash G/K} G/(H \cap {}^g K) \times (e_p\times e_q )$$ with attaching maps $\varphi \times \psi$. Here $e_p\times e_q$ is considered as a cell with dimension $p+q$ by the usual homeomorphism $D^p\times D^q \cong D^{p+q}$.

\begin{remark}
Note that when $G$ is a compact Lie group, this construction is no longer possible. In that case we only have a $(G\times G)$-CW-complex structure on the join $X\ast Y$ and in general it may not be possible to restrict this to a $G$-CW-complex structure via diagonal map $G \to G \times G$.  For compact Lie groups Illman \cite[page 193]{Illman} proves that the join $X\ast Y$ is $G$-homotopy equivalent to a $G$-CW-complex. Also in the above construction it is possible to take $X\ast Y$ with the product topology if one of the complexes $X$ or $Y$ is a finite complex (see \cite[Lemma A.5]{Oliver-Segev}). 
\end{remark}

If $X$ and $Y$ are $G$-CW-complexes, then $(X\ast Y)^H = X^H \ast Y^H$ for every $H \leq G$. Since the join of two Moore spaces is a Moore space it is easy to show that the following holds.
 
\begin{lemma}\label{lem:DimOfJoin} If $X$ is an $\un{n}$-Moore $G$-space and $Y$ is an $\un{m}$-Moore $G$-space, then $X*Y$ is an $\un{k}$-Moore $G$-space, where $\un{k}= \un{n} +\un{m} +1$.
\end{lemma}

\begin{proof} This follows from the usual calculation of homology of join of two spaces and from the above observation on the fixed point subspaces of joins.
\end{proof}

Since it is more desirable to have a dimension function which is additive over the join operation, we define the dimension function for a Moore $G$-space in the following way. 

\begin{definition}\label{def:Dimension} For an $\un{n}$-Moore $G$-space $X$ over $R$, we define the dimension function $\Dim X : \Sub G \to \bbZ$ to be the super class function with values $$(\Dim X) (H)=\un{n} (H) +1$$ for all $H \leq G$.
\end{definition}

By Lemma \ref{lem:DimOfJoin} we have $\Dim (X*Y)=\Dim X + \Dim Y$.
We will take join of a given Moore $G$-space with a homotopy representation. A \emph{homotopy representation} of a finite group $G$ is defined as a $G$-CW-complex $X$ with the property that for each $H \leq G$, the fixed point set $X^H$ is a homotopy equivalent to an $\un{n}(H)$-sphere, where $\un{n} (H)=\dim X^H$. Given a real representation $V$ of $G$, the unit sphere $S(V)$ can be triangulated as a finite $G$-CW-complex and for every $H \leq G$, the fixed point set $S(V)^H=S(V^H)$, so $S(V)$ is a finite homotopy representation with dimension function with values $[\Dim S(V)] (H)=\dim_{\bbR} V^H$. If $X$ is an $\un{n}$-Moore $G$-space, then the join $X\ast S(V)$ is an $\un{m}$-Moore $G$-space, where $\un{m}$ satisfies $\un{m} (H)=\un{n} (H) +\dim _{\bbR} V^H$ for every $H \leq G$. 

\begin{definition} A super class function $f \colon \Sub (G) \to \bbZ$ is called \emph{monotone} if $f(K) \geq f(H)$ for every $K \leq H$. We say $f$ is \emph{strictly monotone} if $f(K) > f(H)$ for every $K<H$.
\end{definition}

We prove the following.

\begin{lemma}\label{lem:strmonotone} Let $X$ be a Moore $G$-space over $R$. Then there is a real $G$-representation $V$ such that $Y=X\ast S(V)$ is a Moore $G$-space with a strictly monotone dimension function and the reduced homologies of $X$ and $Y$ over $R$ are isomorphic.     
\end{lemma}

\begin{proof} Let $s$ be a positive integer and $V$ be $2s$ copies of the regular representation $\bbR G$. Then for each $H \leq G$, we have $\dim_{\bbR} V^H=s |G:H|$. If we choose $s$ large enough, then the dimension function of $Y=X\ast S(V)$ will be strictly monotone. Since the reduced homology of $S(V)$ is isomorphic to $R$ with trivial $G$-action, the reduced homology of $X$ and $Y$ over $R$ are isomorphic.     
\end{proof}

\begin{proposition}\label{pro:ChainHomEq} Let $\bC$ be a finite Moore $R \G_G$-complex of free $R\G_G$-modules. If the dimension function of $\bC$ is strictly monotone, then $\bC$ is chain  homotopy equivalent to a tight Moore $R\G_G$-complex $\bD$ such that $\bD$ is a finite chain complex of free $R\G_G$-modules.   
\end{proposition}

\begin{proof} By applying \cite[Proposition 8.7]{HPY} inductively, as it is done in \cite[Corollary 8.8]{HPY}, we obtain that $\bC$ is chain homotopy equivalent to a tight complex $\bD$. It is clear from the construction that $\bD$ is a finite chain complex of free $k\G _G$-modules.
\end{proof}

Now we are ready to complete the proof of Theorem \ref{thm:main}.

\begin{proof}[Proof of Theorem \ref{thm:main}]
Let $X$ be a finite $\un{n}$-Moore $G$-space over $k$. By Lemma \ref{lem:strmonotone}, we can assume that the function $\un{n}$ is strictly monotone.  Let $\bC=C_* (X^?; k)$ denote the chain complex of $X$ over the orbit category $\G_G=\Or G$. By Proposition \ref{pro:ChainHomEq}, $\bC$ is chain homotopy equivalent to a tight Moore $k\G_G$-complex $\bD$ such that $\bD$ is a finite chain complex of free $k\G _G$-modules. Applying Proposition \ref{pro:main} to the chain complex $\bD$, we conclude that the $n$-th homology of $\bD$, and hence the $n$-th homology of $\bC$, is an endo-permutation $kG$-module generated by relative syzygies.
\end{proof}

\begin{example} The conclusion of Theorem \ref{thm:main} does not hold for a Moore $G$-space $X$ if the fixed point subspace $X^H$ is not a Moore space for some $H\leq G$. One can easily construct examples of Moore $G$-spaces where this happens using the following general construction: Given a Moore $G$-space we can assume $X^G\neq \emptyset$ by replacing $X$ with the suspension $\Sigma X$ of $X$. Given two Moore $G$-spaces $X_1$ and $X_2$ with nontrivial $G$-fixed points, we can take a wedge of these spaces on a fixed point. So, given two Moore $G$-spaces of types $(M_1, n_1)$ and $(M_2, n_2)$, using suspensions and taking a wedge, we can obtain a Moore $G$-space of type $(M_1\oplus M_2, n)$, where $n=lcm(n_1, n_2)$. The direct sum of two endo-permutation $kG$-modules $M_1\oplus M_2$ is not an endo-permutation $kG$-module unless $M_1$ and $M_2$ are compatible. To give an explicit example, we can take $G=C_3\times C_3$ and let $X_1=G/H_1$ and $X_2=G/H_2$ where $H_1$ and $H_2$ are two distinct subgroups in $G$ of index $3$. Then $X=\Sigma X_1 \vee \Sigma X_2$ is a one-dimensional Moore $G$-space with reduced homology $\Delta (X_1)\oplus \Delta (X_2)$, where $\Delta (X_i)=\ker \{kX_i \to k\}$. This module is not an endo-permutation module because $\Delta (X_1)\otimes _k \Delta (X_2)^*$ is not a permutation $kG$-module. One can see this easily by restricting this tensor product to $H_1$ or $H_2$.
\end{example}

Using the same idea, we can construct some other interesting examples of Moore $G$-spaces. 

\begin{example}
The formula in Proposition \ref{pro:Tight} does not hold for an arbitrary $\underline{n}$-Moore $G$-space, it only holds for tight $\un{n}$-Moore $G$-spaces. 
To see this let $G=C_3$, and let $X_1=G/1$ as a $G$-set. Let $X_2 $ be a 1-simplex with a trivial $G$-action on it. Then $X=\Sigma X_1 \vee X_2$ is not a tight complex since $X^G$ is one dimensional, but it is homotopy equivalent to $S^0$.  We can give a $G$-CW-structure to $X$ in such a way that the chain complex for $X$ is of the form
$$ 0 \to k[G/G] \oplus k[G/1] \to \oplus _3 k[G/G]\to 0.$$ Then for this complex, the sum $\sum _{m+1} ^n \Omega _{X_i}$ is zero but the reduced homology module is $\Delta (X_1)$, whose equivalence class in $D^{\Omega } (G)$ is $\Omega _{G/1}$, which is nonzero.
\end{example}

\section{The group of Moore $G$-spaces}
\label{sect:GroupGMoore}

In this section, we define the group of finite Moore $G$-spaces $\cM(G)$ and relate it to the group of Borel-Smith functions $C_b(G)$ and to the Dade group generated by relative syzygies $D^{\Omega} (G)$. 

\begin{definition}\label{defn:equivalence} We say two Moore $G$-spaces $X$ and $Y$ are \emph{equivalent}, denoted by $X \sim Y$, if  $X$ and $Y$ are $G$-homotopy equivalent. By Whitehead's theorem for $G$-complexes, $X$ and $Y$ are $G$-homotopy equivalent if and only if there is a $G$-map $f:X\to Y$ such that for every $H \leq G$, the map on fixed point subspaces $f^H : X^H \to Y^H$ is a homotopy equivalence (see \cite[Corollary II.5.5]{Bredon} or \cite[Proposition II.2.7]{tDieck-BlueBook}). We denote the equivalence class of a Moore $G$-space $X$ by $[X]$.
\end{definition}

It is easy to show that if $X \sim X'$ and $Y\sim Y'$, then $X* Y \sim X' * Y'$. Hence the join operation defines an addition of the equivalence classes of Moore $G$-spaces given by $[X]+[Y]=[X*Y]$. The set of equivalence classes of Moore $G$-spaces with this addition operation is a commutative monoid and we can apply the Grothendieck construction to this monoid to define the group of Moore $G$-spaces.

\begin{definition}\label{defn:GroupMoore}
Let $G$ be a finite $p$-group and $k$ be a field of characteristic $p$. The \emph{group of finite Moore $G$-spaces} $\cM(G)$ is defined as the Grothendieck group of $G$-homotopy classes of finite, capped Moore $G$-spaces with addition given by $[X]+[Y]=[X*Y]$. 
\end{definition}

Since we are only interested in the group of finite Moore spaces, from now on we will assume all Moore $G$-spaces are finite $G$-CW-complexes. Note that every element of $\cM(G)$ is a virtual Moore $G$-space $[X]-[Y]$, and that two such virtual Moore $G$-spaces $[X_1]-[Y_1]$ and $[X_2]-[Y_2]$ are equal in $\cM (G)$ if there is a Moore $G$-space $Z$ such that $X_1*Y_2*Z$ is $G$-homotopy equivalent to $X_2*Y_1*Z$. In particular, for all Moore $G$-spaces $X$ and $Y$ and every real representation $V$, we have $[X]-[Y]=[X*S(V)]-[Y*S(V)]$. Using this we can prove the following:
  
\begin{lemma}\label{lem:StrictlyMonotone}
Every element in $\cM(G)$ can be expressed as $[X]-[Y]$ where $X$ and $Y$ are Moore $G$-spaces over $k$ with strictly monotone dimension functions.  
\end{lemma} 

\begin{proof} Let $[X]-[Y] \in \cM(G)$. Using the argument in the proof of Lemma \ref{lem:strmonotone}, it is easy to see that there is a real representation $V$ such that both $X\ast S(V)$ and $Y\ast S(V)$ have strictly monotone dimension functions. Since $[X]-[Y]=[X\ast S(V)]-[Y\ast S(V)]$, we obtain the desired conclusion.
\end{proof}

Recall that in Definition \ref{def:Dimension}, we defined the dimension function of an $\un{n}$-Moore $G$-space $X$ as the super class function $\Dim X: \Sub (G) \to \bbZ$ satisfying $(\Dim X) (H)=\un{n} (H) +1$ for every $H \leq G$. Note that if $X$ is a tight $\un{n}$-Moore $G$-space, then $\un{n}$ coincides with the geometric dimension function, but in general $\un{n}$ is actually the homological dimension function, giving the homological dimension of fixed point subspaces. Hence $\Dim X$ is well-defined on the equivalence classes of Moore $G$-spaces.

Let $C(G)$ denote the group of all super class functions $f: \Sub (G) \to \bbZ$. Note that $C(G)$ is a free abelian group with rank equal to the number of conjugacy classes of subgroups in $G$. 

\begin{definition}\label{def:DimHom} The assignment $[X] \to \Dim X$ can be extended linearly to obtain a group homomorphism
$$\Dim :  \cM (G) \to C(G).$$
We call the homomorphism $\Dim$ the \emph{dimension homomorphism}.
\end{definition}
 
If $X$ is a finite $G$-set such that $X^G \neq pt$, then  $X$ is a capped Moore $G$-space as a discrete $G$-CW-complex. In this case $\Dim X=\omega _X$ where $\omega_X$ is the element of $C(G)$ defined by
$$ \omega _X (K) =\begin{cases} 1 & \text{if $X^K \neq \emptyset$} \\ 0 & \text{otherwise} \end{cases} $$ 
for every $K \leq G$. Let $\{ e_H \}$ denote the idempotent basis for $C(G)$ defined by $e_H (K)= 1$ if $H$ and $K$ are conjugate in $G$ and zero otherwise. Note that for every $H \leq G$, we have $\omega_{G/H}= \sum _{K \leq _G H } e_K $, where the sum is over all $K$ such that $K^g \leq H$ for some $g\in G$. Since the transition matrix between $\{\omega_{G/H}\}$ and $\{ e_K\}$ is an upper triangular matrix with 1's on the diagonal, the set $\{ \omega _{G/H} \}$, over the set of conjugacy classes of subgroups $H \leq G$, is a basis for $C(G)$. We conclude the following.
 
\begin{lemma}\label{lem:basis} The set of super class functions $\{\omega _{G/H}\}$ over all transitive $G$-sets $G/H$ is a basis for $C(G)$. Moreover, for every $H \leq G$, we have $e_H = \sum _{K \leq _G H} \mu _G (K, H) \omega _{G/K}$ where $\mu _G (K, H)$ denotes the M\" obius function of the poset of conjugacy classes of subgroups of $G$.  
\end{lemma}

\begin{proof} See \cite[Lemma 2.2]{Bouc-Aremark}.  
\end{proof}

The following is immediate from this lemma.

\begin{proposition}\label{pro:DimSurj} The dimension homomorphism $\Dim : \cM (G) \to C(G)$ that takes $[X]$ to its dimension function $\Dim X$ is surjective.
\end{proposition}

\begin{proof} 
For $H\leq G$, let $X=G/H$ if $H\neq G$, and $X=G/G \coprod G/G$ if $H=G$. Then, $X$ is a capped Moore $G$-space and  $\dim X=\omega_{G/H}$. Hence by Lemma \ref{lem:basis}, the map $\Dim : \cM(G) \to C(G)$ is surjective.
\end{proof}

Now consider the group homomorphism  $\Hn: \cM(G) \to D^{\Omega} (G)$, defined as the linear extension of the assignment which takes an isomorphism class $[X]$ of an $\un{n}$-Moore $G$-space $X$ to the equivalence class of the $n$-th reduced homology $[\widetilde H_n (X, k)]$ in $D^{\Omega} (G)$, where $n=\un{n} (1)$. This extension is possible because the assignment $[X] \to \Hn ([X])$ is additive. Note that if $[X_1]-[Y_1]=[X_2]-[Y_2]$, then there is a Moore $G$-space $Z$ such that $$X_1*Y_2*Z\cong X_2*Y_1*Z.$$ This gives that $\Hn ([X_1])-\Hn([Y_1])=\Hn ([X_2])-\Hn ([Y_2])$ in $D^{\Omega}(G)$. So, $\Hn$ is a well-defined homomorphism.  

There is a third homomorphism $\Psi : C(G) \to D^{\Omega}(G)$ which can be uniquely defined as the group homomorphism which takes $\omega _{G/H}$ to $\Omega _{G/H}$.   In \cite[Theorem 1.7]{Bouc-Aremark}, it is also proved that $\Psi$ takes $\omega _X$ to $\Omega _X$ for every $G$-set $X$, by showing that the relations satisfied by $\omega _{X}$ also hold for $\Omega_{X}$. Now we state the main theorem of this section:

\begin{proposition}\label{pro:DependsOnly} Let $\Psi : C(G) \to D^{\Omega} (G)$ denote the homomorphism defined by $\Psi (\omega _{G/H})=\Omega _{G/H}$, and let $\Hn : \cM (G) \to D^{\Omega } (G)$ be the homomorphism which takes the equivalence class of a finite, capped Moore $G$-space $[X]$ to the equivalence class $[H_n (X; k)]$ in $D^{\Omega}(G)$. Then $$\Hn= \Psi \circ \Dim.$$ In particular, if $X$ is a finite, capped $\un{n}$-Moore $G$-space, then the equivalence class $[H_n (X; k)]$ in $D^{\Omega } (G)$ depends only on the function $\un{n}$.
\end{proposition}

To prove Proposition \ref{pro:DependsOnly}, we need to introduce a property that is found in chain complexes of $G$-simplicial complexes. Let $X$ be a finite $G$-simplicial complex and let $\bC:=\widetilde C_* (X^?; k )$ denote the reduced chain complex of $X$ over the orbit category $\G_G$. The complex $\bC$ is a finite chain complex of free $k\G _G$-modules. Let $X_i$ denote the $G$-set of $i$-dimensional simplices in $X$ for every $i$. Note that since $X$ is a $G$-simplicial complex, the collection $\{X_i\}$ satisfies the following property:

\vskip 5pt 

$(**)$ For every subgroup $H \leq G$, if $X_i^H \neq \emptyset$ for some $i$, then $X_j ^H \neq \emptyset$ for every $j \leq i$. 

\vskip 5pt

Note that if $X$ is a $G$-CW-complex and $X_i$ denotes the $G$-set of $i$-dimensional cells in $X$, then the collection $\{X_i\}$ does not satisfy this property in general. If this property holds for a $G$-CW-complex $X$, then we say $X$ is a \emph{full} $G$-CW-complex. More generally, we define the following:

\begin{definition} Let $\bC$ be a finite free chain complex of $R\G _G$-modules. For each $i$, let $X_i$ denote the $G$-set such that $\bC_i \cong R[X_i ^?]$. We say $\bC$ is a \emph{full $R\G_G$-complex} if the collection of $G$-sets $\{X_i\}$ satisfies the property $(**)$. 
\end{definition}

For chain complexes that are full, we have the following observation:

\begin{lemma}\label{lem:DimSum} Let $\bC$ be a finite chain complex $k\G_G$ of dimension $n$. Suppose that $\bC$ is a full complex, and $X_i$ denotes the $G$-set such that $\bC_i \cong R[X_i^?]$ for each $i$.  Let $f$ be the super class function defined by $f(H)=\dim \bC(H) +1$ for all $H \leq G$. Then $$f=\sum _{i=0} ^n \omega _{X_i}.$$
\end{lemma}

\begin{proof} Let $H \leq G$. The sum $\sum _i \omega _{X_i} (H)$ is equal to the number of $i$ such that $X_i ^H \neq \emptyset$. Since $X_i ^H \neq \emptyset$ if and only if $i$ satisfies $0 \leq i \leq \dim X^H$, we obtain that $\sum _i \omega _{X_i} (H)=\dim \bC(H) +1$.
\end{proof} 

Now, we are ready to prove Proposition \ref{pro:DependsOnly}. 

\begin{proof}[Proof of Proposition \ref{pro:DependsOnly}] Let $[X]-[Y]$ be an element in $\cM(G)$. By Lemma \ref{lem:StrictlyMonotone}, we can assume $X$ and $Y$ are Moore $G$-spaces with strictly monotone dimension functions. Moreover we can assume that both $X$ and $Y$ are $G$-simplicial complexes. This is because every $G$-CW-complex is $G$-homotopy equivalent to a $G$-simplicial complex (see \cite[Proposition A.4]{Oliver-Segev}). Let $\bC:=\widetilde C_* (X ^?; k)$ denote the reduced chain complex for $X$ over the orbit category. 

Since the dimension function of $X$ is strictly monotone, by Proposition \ref{pro:ChainHomEq} the complex $\bC$ is chain homotopy equivalent to a tight complex $\bD$, which is by construction a finite chain complex of free $k\G _G$-modules. Moreover, we can take $\bD$ to be a full complex. To see this, observe that since $X$ is a simplicial complex, $\bC$ is a full complex. The construction of $\bD$ involves erasing chain modules of $\bC$ above the homological dimension, hence we can assume that $\bD$ is also a full complex.  

For each $i$, let $X_i$ denote the finite $G$-set such that $\bD _i \cong k[X_i ^?]$. By Lemma \ref{lem:DimSum}, for each $H \leq G$, we have $\dim \bD(H) +1=\sum _{i=0} ^n \omega _{X_i}(H)$, where $\dim \bD (H)$ denotes the chain complex dimension of $\bD$. Since $\bD$ is a tight complex, we have $\un{n} (H) =\dim \bD(H)$ for all $H \leq G$, so we have $\Dim \bD=\sum _{i=0} ^n \omega _{X_i} $. By Proposition \ref{pro:computing}, the equation
$$[\widetilde H_{n} (X; k)]=\sum _{i=m+1}^n \Omega _{X_i}$$ holds in the Dade group $D^{\Omega} (G)$, where $m=\un{n} (G)$. Note that since $\bD$ is a full complex, $X_i^G \neq \emptyset$ for all $i \leq m$. This means that $\Omega _{X_i}=0$ for all $i \leq m$, hence we conclude $$[\widetilde H_{n} (X; k)]=\sum _{i=0}^n \Omega _{X_i}=\Psi (\sum _{i=0} ^n \omega_{X_i})=\Psi (\Dim \bD).$$  The same equality holds for $[Y]$, hence $\Hn= \Psi \circ \Dim$. 
\end{proof}

The rest of the section is devoted to the proof of Theorem \ref{thm:mainSeq} stated in the introduction. Let $\cM_0(G)$ denote the kernel of the homomorphism $\Hn: \cM(G) \to D^{\Omega} (G)$. Note that $\Hn$ is surjective because $\Hn ([X])=\Omega _X$ when $X$ is a finite $G$-set such that $|X^G| \neq 1$.  By Proposition \ref{pro:DependsOnly}, we have $\Hn=\Psi \circ \Dim$, so the map $\Psi$ is also surjective. Hence there is a commuting diagram 
$$\xymatrix{
0 \ar[r] & \cM_0(G) \ar[r] \ar[d]^{\Dim_0} &  \cM(G) \ar[d]^{\Dim} \ar[r]^{\Hn}
& D^{\Omega} (G)  \ar[d]^{=}  \ar[r] & 0  \\
0  \ar[r] & \ker \Psi \ar[r] &  C(G) \ar[r]^{\Psi}
&  D^{\Omega} (G) \ar[r] & 0 }\\ $$
where the horizontal sequences are exact. By Proposition \ref{pro:DimSurj}, the homomorphism $\Dim$ is surjective, hence by the Snake Lemma $\Dim _0$ is also surjective. To complete the proof of Theorem \ref{thm:mainSeq}, it remains to show that $\ker \Psi$ is equal to the group of Borel-Smith functions $C_b(G)$. Recall that Borel-Smith functions are super class functions satisfying certain conditions called Borel-Smith conditions. A list of these conditions can be found in \cite[Definition 3.1]{Bouc-Yalcin} or \cite[Definition 5.1]{tDieck-BlueBook}.

%
%
%
%

\begin{proposition} Let $G$ be a $p$-group. Then, the kernel of the homomorphism $\Psi : C(G) \to D^{\Omega }(G)$ is equal to the group of Borel-Smith functions $C_b(G)$.
\end{proposition}

\begin{proof} A proof of this statement can be found in \cite[Theorem 1.2]{Bouc-Yalcin}, but the proof given there uses the biset functor structure of the morphism $\Psi : C(G) \to D^{\Omega} (G)$, hence the tensor induction formula of Bouc. Here we give an argument independent of the tensor induction formula.

Let $f \in C_b (G)$ be a Borel-Smith function. Then by Dotzel-Hamrick \cite{Dotzel-Hamrick} there is a virtual real representation $U-V$ such that $\Dim U -\Dim V=f$. Since the unit spheres $S(U)$ and $S(V)$ are orientable homology spheres over $k$, both $S(U)$ and $S(V)$ are Moore $G$-spaces and the element $[S(U)]-[S(V)]$ is in $\cM_0 (G)$. This proves that $f\in \im(\Dim _0)=\ker \Psi$.  

For the converse, let $f=\Dim _0 ([X]-[Y])$ for some $[X]-[Y] \in \cM_0 (G)$. Then, $\Hn ([X])=\Hn ([Y])$. We want to show that $f$ satisfies the Borel-Smith conditions. Since the Borel-Smith conditions are given as conditions on certain subquotients, first note that for any subquotient $H/L$, we can look at $(H/L)$-spaces $X^L$ and $Y^L$, and the dimension function for the virtual Moore $G$-space $[X^L]-[Y^L]$ would satisfy Borel-Smith conditions if and only if the function $f$ satisfies the Borel-Smith condition corresponding to the subquotient $H/L$. For these subquotients it is easy to check that every super class function $f$ in $\ker \Psi$ satisfies the Borel-Smith conditions (see \cite[Page 12]{Bouc-Yalcin}). 
\end{proof} 

This completes the proof of Theorem \ref{thm:mainSeq}. We can view this theorem as a topological interpretation of the exact sequence given in \cite[Theorem 1.2]{Bouc-Yalcin}.  There is an interesting corollary of Theorem 1.4 that gives a slight generalization of the Dotzel-Hamrick Theorem (see \cite{Dotzel-Hamrick}).  

\begin{proposition}\label{pro:BorelSmith}  Let $G$ be a finite $p$-group and let $k$ be a field of characteristic $p$. Suppose that $X$ is a finite Moore $G$-space of dimension $n$ such that $H_n (X; k)$ is a capped permutation $kG$-module. Then the super class function $\Dim X$ satisfies the Borel-Smith conditions.
\end{proposition}

\begin{proof} By the assumption on homology, $[H_n(X;k)]=0$ in $D^{\Omega}(G)$, hence $[X] \in \cM_0 (G)$. Now the result follows from Theorem \ref{thm:mainSeq}.
\end{proof}

\section{Operations on Moore $G$-spaces}
\label{sect:Operations}

The main aim of this section is to show that the assignment $G \to \cM(G)$ defined over a collection of $p$-groups $G$ has an easy to describe biset functor structure and that the maps $\Hn$ and $\Dim$ are both natural transformations of biset functors.   We also give a topological proof of Bouc's tensor induction formula for relative syzygies (see Theorem \ref{thm:TensorInd} below for a statement).  

Let $\cC$ denote a collection of $p$-groups closed under taking subgroups and subquotients, and let $R$ be a commutative ring with unity. An $(H, K)$-biset is a set $U$ together with a left $H$-action and a right $K$-action such that $(hu)k=h(uk)$ for every $h\in H$, $u\in U$, and $k\in K$. The $\cC$-biset category over $R$ is the category whose objects are $H \in \cC$ and whose morphisms $\Hom (K,H)$ for $H, K\in \cC$ are given by $R$-linear combinations of $(H,K)$-bisets, where the composition of two morphisms is defined by the linear extension of the assignment $(U, V) \to U \times_{K} V$ for $U$ an $(H,K)$-biset and $V$ a $(K,L)$-biset. A biset functor $F$ on $\cC$ over $R$ is a functor $F$ from the $\cC$-biset category over $R$ to the category of $R$-modules. We refer the reader to \cite{Bouc-BisetBook} for more details on biset functors for finite groups.

To define a biset functor structure on $\cM(-)$, we need to define the action of an $(H,K)$-biset $U$ on an isomorphism class $[X]$ in $\cM (H)$ and extend it linearly. To simplify the notation we will define these actions on a representative of each equivalence class and show that the definition is independent of the choice of the representative. Every $(H,K)$-biset can be expressed as a composition of 5 types of basic bisets, called restriction, induction, isolation, inflation, and deflation bisets (see \cite[Lemma 2.3.26]{Bouc-BisetBook}). Except for the induction biset, the action of a biset on a Moore $G$-space is easy to define. For example, if $\varphi : H \to K$ is a group homomorphism and $U=(H\times K)/\Delta(\varphi)$ where $\Delta (\varphi)=\{(h, \varphi(h)) \, | \, h\in H \}$, then for a Moore $K$-space $X$, we define $\cM(U) (X)$ as the Moore space $X$ together with the $H$-action given by $hx=\varphi (h)x$. This gives us the action of restriction, isolation, and inflation bisets on a Moore space. The action of the deflation biset can also be defined easily by taking fixed point subspaces: Given a normal subgroup $N$ in $G$ and a Moore $G$-space $X$, we define $\Def ^G _{G/N} X$ as the $G/N$-space $X^N$. 

The action of the induction biset $U={}_H (H)_K$, where $K \leq H$ is harder to define. This operation is called \emph{join induction} and the difficulty comes from describing the equivariant cell structure of the resulting space. To make this task easier we will define join induction on a Moore $G$-space whose $G$-CW-structure comes from a realization of a $G$-poset. 

Recall that a $G$-poset is a partially ordered set $X$ together with a $G$-action such that $x\leq y$ implies $gx \leq gy$ for all $g\in G$. Associated to a $G$-poset $X$, there is a simplicial $G$-complex whose $n$-simplices are given by chains of the form $x_0 < x_1 <\cdots < x_n$ where $x_i \in X$. This simplicial complex is called the associated flag complex of $X$ (or the order complex of $X$) and is denoted by $\Flag (X)$. We denote the geometric realization of $\Flag (X)$ by $|X|$. The complex $\Flag (X)$ is an admissible simplicial $G$-complex, i.e., if it fixes a simplex, it fixes all its vertices. Since $\Flag (X)$ is an admissible $G$-CW-complex, the realization $|X|$ has a $G$-CW-complex structure. For more details on $G$-posets we refer the reader to \cite[Definition 11.2.7]{Bouc-BisetBook}. 

By an equivariant version of the simplicial approximation theorem, every $G$-CW-complex is $G$-homotopy equivalent to a simplicial $G$-complex (see \cite[Proposition A.4]{Oliver-Segev}). Given a $G$-simplicial complex $X$, by taking the barycentric subdivision, we can further assume that $X$ is the flag complex of the poset of simplices in $X$. Therefore, up to $G$-homotopy we can always assume that a given Moore $G$-space is the realization of a $G$-poset $X$.

Let $X$ and $Y$ be two $G$-posets. The product of $X$ and $Y$ is defined to be the $G$-poset $X\times Y$ where the $G$-action is given by the diagonal action $g(x,y)=(gx, gy)$, and the order relation is given by $(x,y) \leq (x', y')$ if and only if $x\leq x'$ and $y \leq y'$. The join of two $G$-posets $X$ and $Y$ is defined to be the disjoint union $X\coprod Y$ together with extra order relations $x \leq y$ for all $x\in X$ and $y \in Y$. However this description of the join is not suitable for defining join induction since it is not symmetric. Instead we use the following definition:

\begin{definition}\label{def:join} For a $G$-poset $X$, let $cX$ denote the \emph{cone of $X$} where $cX=\{0_X \} \coprod X$ with trivial $G$-action on $0_X$. The order relations on $cX$ are the same as the order relations on $X$ together with an extra relation $0_X \leq x$ for all $x\in X$. We define the (symmetric) join of two $G$-posets $X$ and $Y$ as the poset defined by $$X*Y := (cX\times cY)-\{(0_X, 0_Y)\}.$$
\end{definition}

Throughout this section, when we refer to the join of two posets, we always mean the symmetric join defined above. The realization $|X*Y|$ of the (symmetric) join of two $G$-posets $X$ and $Y$ is $G$-homeomorphic to the join $|X|*|Y|$ of realizations of $X$ and $Y$. This is proved in \cite[Proposition 1.9]{Quillen}) but below we prove this more generally for the join of finitely many $G$-posets. We define
the (symmetric) join $X_1 \ast \cdots \ast X_n$ of $G$-posets $X_1, X_2, \dots , X_n$ as the $G$-poset $\prod _i cX_i -\{(0_{X_1}, \dots, 0_{X_n} )\}$ with the diagonal $G$-action. Note that the geometric join $|X_1|\ast \cdots \ast |X_n| $ can be identified with the subspace of $\prod_i c|X_i|$ formed by elements $t_1x_1+\cdots + t_n x_n$ such that $t_i \in [0,1]$ and $\sum _i t_i =1$. Here $c|X_i|$ denotes the identification space $X\times [0,1]/\sim$ where $(x,0)\sim (x', 0)$ for all $x,x'\in X$. Alternatively we can consider elements of $c|X_i|$ as expressions $t_ix_i$ where $t_i\in [0,1]$ and $x_i \in |X_i|$. We have the following observation.
 
\begin{proposition}\label{pro:Realization} Let $\{X_i\}$ be a finite set of $G$-posets. Then $|\ast _i X_i|$ is $G$-homeomorphic to the (topological) join $\ast_i|X_i|$ of the realizations of the $X_i$. 
\end{proposition}

\begin{proof} The realization $|\ast X_i|=|\prod _i cX_i -\{(0_{X_1}, \dots, 0_{X_n} )\}|$ can be identified with the union of subspaces $$ \bigcup _i c|X_1| \times \cdots \times |X_i| \times \cdots \times c|X_n|$$
in the product $\prod_i c|X_i|$. Using the radial projection from the point $\{(0_{X_1}, \dots, 0_{X_n} )\}$ in $\prod_i c|X_i|$, we can write a $G$-homeomorhism between this subspace and the geometric join $|X_1|\ast \cdots \ast |X_n| $. This homeomorphism takes the point $(t_1x_1, \dots , t_n x_n)$ in $\prod _i c|X_i| $ to $(t_1'x_1, \dots , t_n 'x_n)$ where $t_i '=t_i/(\sum _i t_i)$ for all $i$. Note that this argument only works if we take the compact open topology on the product, not the product topology (see \cite[Theorem 3.1]{Walker}). \end{proof}

Let $U$ be a finite $(H,K)$-biset and $X$ be a $K$-poset. Define $t_U (X)$ as the set $$t_U (X) :=\Map _K (U^{op}, X)$$ of all functions $f : U \to X$ such that $f(uk)=k^{-1} f(u)$. The poset structure on $t_U (X)$ is defined by declaring $f_1\leq f_2$ if and only if for every $u\in U$, $f_1(u) \leq f_2 (u)$. There is an $H$-action on $t_U (X)$ given by $(hf) (u)=f(h^{-1} u)$ for all $h \in H$, $u \in U$. The set $t_U (X)$ is an $H$-poset with respect to this action. The assignment $X\to t_U (X)$ is called the generalized tensor induction of posets associated to $U$ (see \cite[11.2.14]{Bouc-BisetBook}).

\begin{definition}\label{defn:JoinInd} Let $K$ and $H$ be finite groups and $U$ be a finite $(H,K)$-biset. For a $K$-poset $X$, we define the \emph{join induction induced by $U$} on $X$ as the $H$-poset $$\jn _U X:=t_U (cX) -\{ f_{0} \} $$ where $f_{0}$ is the constant function defined by $f_{0} (u)=0_X$ for all $u\in U$. When $U={}_H H _K $ is the induction biset, then we denote the join induction operation $\jn_U $ by $\jn_ K ^H$, and we call it \emph{join induction from $K$ to $H$}. \end{definition}

The following result justifies this definition.

\begin{proposition}\label{pro:Justifies} Let $R$ be a coefficient ring and let $X$ be a $K$-poset such that the realization $|X|$ is a Moore $K$-space over $R$. Then for every $(H,K)$-biset $U$, the realization of the $H$-poset $\jn _U X$ is a Moore $H$-space over $R$.
\end{proposition}

\begin{proof}  We need to show that for every $L\leq H$, the fixed point subspace $|\jn _U X | ^L=|(\jn _U X )^L|$ is a Moore space over $R$. We have 
\begin{equation*} 
\begin{split} 
(\jn _U X) ^L = \Hom _H (H/L,\ t_U (cX) )-\{ f_0\} = \Hom _K ( U ^{op} \times _H (H/L),\ cX)-\{ f_0 \}. 
\end{split}
\end{equation*}
By \cite[Lemma 11.2.26]{Bouc-BisetBook}, we have $$U ^{op} \times _H (H/L)\cong \coprod_{u \in L \backslash U/K } K/ L^u$$ where $L^u$ is the subgroup of $K$ defined by $L^u=\{ k \in K \colon uk=lu \ \text{for some}\ l\in L \}$.
Using this we obtain
\begin{equation*}
\begin{split} 
(\jn _U X) ^L  & = \Hom _K (
\coprod _{u \in L \backslash U/K } K/ L^u ,\ cX)- \{ f_0 \} =(\prod _{u \in L \backslash U /K } (cX)^{L ^u} )-\{(0_X, \dots, 0_X)\} 
\end{split}
\end{equation*}
Applying Proposition \ref{pro:Realization}, we conclude  
\begin{equation}\label{eqn:FPFormula} 
|\jn _U X|^L = \ast _{u\in L \backslash U /K} |X|^{L^u }.
\end{equation}
Since the join of a collection of Moore spaces is a Moore space,  $|\jn _U X| ^L $ is a Moore space over $R$. 
\end{proof}

Let $U$ be an $(H,K)$-biset. If $f : X \to X'$ is a $K$-poset map that induces a $K$-homotopy equivalence $|f|: |X|\to |X'|$, then $|\jn _U f|: |\jn _U X|\to |\jn _U X' | $ is an $H$-homotopy equivalence. To see this, observe that for every $L \leq H$ the induced map $$|\jn_U f|^L: |\jn_U X|^L  \to |\jn_U X'|^L$$ is a homotopy equivalence by the fixed point formula given in Equation \ref{eqn:FPFormula}. Hence by Whitehead's theorem for $G$-complexes, $|\jn_U f|$ is a $G$-homotopy equivalence. This proves that $\jn_U X$ defines a well-defined map on the equivalence classes of Moore $K$-spaces.

Given two $K$-posets $X$ and $Y$, we have $$\jn_U (X\ast Y)= t_U (c(X\ast Y) )-\{ f_0\} = t_U (cX\times cY ) -\{ f_0\} =\jn _U (X) \ast \jn _U (Y).$$  Hence we conclude that $\jn _U$ induces a well-defined group homomorphism $$\cM(U) \colon \cM(K) \to \cM (H)$$ for every $(H,K)$-set $U$. We show below that this operation satisfies all the necessary conditions for defining a biset functor.

\begin{proposition}\label{pro:BisetM} There exists a biset functor $\cM$ over $p$-groups such that for any $p$-group $G$, $\cM(G)$ is the group of finite Moore $G$-spaces, and for any $(H,K)$-biset $U$, $$\cM(U) :\cM(K)\to \cM (H)$$ is the group homomorphism induced by the generalized join induction $\jn _U$.
\end{proposition}

\begin{proof} Let $U$ be a $(H,K)$-biset and $V$ be an $(L, H)$-biset. For every $K$-poset $X$,  $$\jn _V (\jn _U X)= t_V (c(\jn _U X ))-\{ f_{0} \} = t_V (t_U (cX) ) -\{f _{0} \}.$$ Since $t_V (t_U (cX))=t_{V \times _H  U} (cX)$ (see \cite[Proposition 11.2.20]{Bouc-BisetBook}), we conclude that
$$\jn _V (\jn _U X)=\jn _{V \times _H U} X.$$ 
By a similar argument, we see that if $U$ and $U'$ are two $(H,K)$-bisets, then $\jn _{U\amalg U'} X =\jn _U X \ast \jn _{U'} X$ for every Moore $K$-space $X$. Hence $\cM(U) : \cM (K) \to \cM(H)$ defines a biset functor $\cM(-)$ on any collection $\cC$ of $p$-subgroups that is closed under conjugations and taking subquotients. 
\end{proof}

Now we will show that the dimension homomorphism $\Dim : \cM (G) \to C(G)$ is a natural transformation of biset functors. Note that the biset action on the group of super class functions is defined as the dual of the biset action on the Burnside ring. Let $B(G)$ denote the Burnside ring of the group $G$ and $B^* (G)=\Hom (B(G), \bbZ)$. We can identify $C(G)$ with $B^* (G)$ by sending $f$ to a homomorphism that takes the transitive $G$-set $G/H$ to $f(H)$. Under this identification we often write $f(G/H)$ for $f(H)$ when we want to think of $f$ as an element of $B^*(G)$.

For a $(H,K)$-biset $U$, the action of $U$ on $B^*(G)$ is defined as the dual of the $U$-action on $B(G)$. In particular, for $f \in B^* (K)$, we define the generalized induction $\Jnd _U f$ as the super class function that satisfies 
\begin{equation}\label{eqn:JndFormula} 
(\Jnd_U f) (H/L)= f(U^{op} \times _H (H/L))=\sum  _{u\in L \backslash U /K} f(K/L^u)
\end{equation} 
for every $L \leq H$ (see \cite[Page 7]{Bouc-Yalcin} for more details).

\begin{proposition}\label{pro:DimCommutes} Let $U$ be a $(H,K)$-biset and let $X$ be a Moore $K$-space. Then $\jn _U X$ is a Moore $H$-space with dimension function $\Dim (\jn _U X)=\Jnd _U (\Dim X).$  Hence the dimension homomorphism $\Dim : \cM(-) \to C(-)$ is a natural transformation of biset functors.
\end{proposition}

\begin{proof} Let $L \leq H$. Using the formulas in Equations \ref{eqn:FPFormula} and \ref{eqn:JndFormula}, we obtain
\begin{equation*}
[\Dim (\jn _U X)] (H/L)=\sum  _{u\in L \backslash U /K} (\Dim X) (K/L^u) =[\Jnd _U (\Dim X )](H/L).
\end{equation*}
\end{proof}

Now we consider the homomorphism $\Hn: \cM(G)\to D^{\Omega} (G)$. This homomorphism also defines a natural transformation of biset functors, but for this we first need to explain the biset functor structure on $D^{\Omega} (-)$. Again it is relatively easy to define the action of an $(H,K)$-biset on an endo-permutation $kK$-module $M$ when the biset $U$ is a restriction, inflation, isogation, or deflation biset. The action of the induction biset is harder to define and it is done using tensor induction of $kG$-modules. To define tensor induction operation on $D^{\Omega} (G)$ one needs to prove that tensor induction of a relative syzygy is generated by relative syzygies. To prove this, Bouc \cite{Bouc-Tensor} proved the tensor induction formula stated below as Theorem \ref{thm:TensorInd}. We will first give a topological proof for this formula.

Let $K$ be a subgroup of $H$, and let $\{h_1 K, \dots, h_s K\}$ be a set of coset representatives for cosets of $K$ in $H$. For each $h\in H$, there is a permutation $\pi$ of the set $\{1, \dots, s\}$ such that for every $i\in \{ 1, \dots, s\}$, $$hh_i =h_{\pi(i) } k_i$$ for some $k_i \in K$. The tensor induction $\Ten _K ^H M$ is defined as the $kH$-module which is equal to the tensor product $M\otimes_k \cdots \otimes_k M$ ($s$ times) as a $k$-vector space, and the $H$-action is defined by $$h (m_1\otimes \cdots \otimes m_s) =k_{\pi ^{-1} (1)} m_{\pi^{-1} (1)} \otimes \cdots \otimes k_{\pi ^{-1} (s)}  m_{\pi ^{-1} (s)} $$
for every $h\in H$. The tensor induction distributes over tensor products $$\Ten _K ^H (M_1 \otimes M_2)=\Ten _K ^H M_1 \otimes \Ten _K ^H M_2$$ and there is a Mackey formula for tensor induction similar to the Mackey formula for additive induction (see \cite[Proposition 3.15.2]{Benson-Book1}).  

\begin{proposition}\label{pro:HomologyJoinInd} Let $K \leq H$ be $p$-groups, and let $X$ be a Moore $K$-space over $k$ with nonzero reduced homology at dimension $n$. Then the reduced homology of $\jn _K ^H  X$ is isomorphic to $\Ten_K ^H (\widetilde H_n (X; k))$. \end{proposition} 

\begin{proof} As before we can assume that the Moore space $X$ is a realization of a $G$-poset $X$ and that join induction is defined by $\jn _K ^H X=\Hom _K (H, cX)-\{ f_0 \}$. Let $\{ h_1K, \dots, h_sK\}$ be a set of left coset representatives of $K$ in $H$. We can consider the $H$-poset $\Hom _K (H, cX)$ as a product of $K$-posets $\prod _{i=1} ^s cX$ with the $H$-action given by $$h(x_1,\dots, x_s)=(k_{\pi ^{-1} (1)} x_{\pi^{-1} (1)}, \dots, k_{\pi ^{-1} (s)} x_{\pi ^{-1} (s)} )$$ for every $h \in H$.  The permutation $\pi$ and the elements $k_1, \dots, k_s \in K$ are defined as in the definition of tensor induction. As in the proof of Proposition \ref{pro:Realization}, we can identify the realization of this poset with the geometric join of $s$ copies of the realization of $X$. This gives a simplicial complex, also denoted $\jn _K ^H X$, where the simplices of $\jn _K ^H X$ are unions of simplices of the flag complex $\Flag(X)$. Hence a simplex of $\jn _K ^H X$ is of the form $\sigma =\cup _{i=1} ^s \sigma _i$ where  $\sigma _i $ is a simplex of $\Flag (X)$ or an empty set, for every $i\in \{1, \dots, s\}$. The $H$-action on $\sigma$ is given by $$h\sigma=\bigcup _{i=1} ^s k_{\pi^{-1} (i) } \sigma _{\pi ^{-1}(i)}$$
for every $h \in H$. Note that the $H$-action on $\jn_K ^H X$ is not admissible, i.e., an element $h \in H$ can take a simplex to itself without fixing it pointwise. We can take the barycentric subdivision of $\jn _K ^H X$ to obtain an admissible complex. Since an admissible $H$-complex has a natural $H$-CW-complex structure, this will give an $H$-CW-complex structure on $\jn_K ^H X$.
 
 We can calculate the homology of $\jn_K ^H X$ using the cellular homology and the $H$-CW-complex structure that we just described. But it is not easy to find an $H$-basis for the cellular chain complex of this $H$-CW-complex. Instead we will use the chain complex of the simplicial complex $\jn_K^H X$ before taking the barycentric subdivision. This gives us a signed permutation complex instead of a permutation complex, but it is easier to give a basis for $H$-orbits of cells. Note that if $\widetilde C (X; k)$ denotes the reduced (simplicial) chain complex of the flag complex $\Flag (X)$, then the chain complex for $\jn_K ^H X$ will be the complex
 $\Ten _K ^H \widetilde C_* (X; k)$. The tensor induction of chain complexes is defined in a similar way to the tensor induction of modules, except that there is a twisting coming from the sign convention for chain complexes. The action of $h\in H$ on $\Ten _K ^H \widetilde C (X; k)$ is defined by $$h (\sigma_1 \otimes \cdots \otimes \sigma _s)=(-1) ^{\nu} k_{\pi^{-1} (1) } \sigma _{\pi^{-1} (1)} \otimes \cdots \otimes k_{\pi^{-1} (s) } \sigma _{\pi ^{-1} (s)}$$
 where $$\nu=\sum _{\substack{i<j\\ \pi(i) >\pi (j) }} \deg(\sigma _i) \deg(\sigma_j).$$
More details about this formula can be found in \cite[Sect. 4.1]{Benson-Book2}. In the chain complex $C_* (X; k)$ the degree of a simplex $\sigma: x_0<\cdots < x_n$ is $n$, but in the above formula we take the reduced complex $\widetilde C (X; k)$ as a chain complex starting from degree zero, i.e., as a complex shifted by one degree. In other words in the above formula, the degree of $\sigma$ is taken as $n+1$, i.e., as the number of vertices in $\sigma$. This is consistent with the definition of the orientation of the simplex $\sigma=\cup _i \sigma_i$.  

To complete the proof note that we can assume that the complex $C_*=\widetilde C_* (X; k)$ is a tight complex. If not, we can replace it with a tight complex up to chain homotopy. Note that tensor induction of chain homotopic complexes are chain homotopic (see \cite[Lemma 4.1.1]{Benson-Book2}). Therefore we can assume $C_i=0$ for $i>n$, hence there is $K$-map $H_n (X; k) \to C_n (X; k)$. We can consider this map as a chain map  $H_n (X;k) \to C_* $ from a chain complex concentrated at degree $n+1$. This gives an $H$-map
$$ \Ten _K ^H H_n (X; k) \otimes_k k^{(n+1)} \to \Ten _K ^H C_*$$
that induces an isomorphism on homology by the K\" unneth theorem. The module $k^{(n+1)}$ is a one dimensional module which is trivial if $n+1$ is even, and it is the sign representation if $n+1$ is odd. Since either the characteristic of $k$ is even or $H$ is a  $p$-group for an odd prime $p$, in our case the module $k^{(n+1)}$ is isomorphic to $k$. So, we obtain the desired isomorphism.  
\end{proof}

\begin{example} Over a ring $R$ the conclusion of Proposition \ref{pro:HomologyJoinInd} is only true up to a twist coming from sign representation. To see this consider the following situation. Let $H=C_2$, $K=1$, and $X=S^0$ be the zero dimensional sphere. The realization of the poset $\Hom _K (H, cX)-\{ f_0 \}$ can be pictured as the barycentric subdivision of a square, with corners $(\pm 1, \pm 1)$ in $\bbR ^2$. The $H$-action is defined by reflection $(x,y)\to (y,x)$. As a simplicial complex, this complex has 8 vertices and 8 edges. As we did in the proof of Proposition \ref{pro:HomologyJoinInd}, we can consider the geometric join, which will be the square inside the original square with corners given by $(\pm 1,0)$ and $(0, \pm 1)$. Now as a simplicial complex this complex has 4 vertices and 4 edges and some of the edges are taken to themselves with reverse orientation by the $H$-action. For example the edge between $(0,1)$ and $(1,0)$ is taken to itself. So the $H$-action is not admissable. The reduced chain complex for this simplicial complex is of the form
$$ 0 \to \bbZ[H/1] \otimes (\oplus _2 \widetilde\bbZ ) \to \oplus _2 \bbZ [H/1] \to \bbZ \to 0$$ where $\widetilde \bbZ$ denotes the integers with $-1$ action of $C_2$.  Note that this complex is the tensor induction of the chain complex $\widetilde C_* (S^0)$ which is a chain complex of the form $$0 \to \oplus _2 \bbZ \to \bbZ \to 0.$$ Here we consider the generators of the module $\oplus _2 \bbZ$ as degree one chains.
Note that the reduced homology of the complex $S^0$ is the trivial module $\bbZ$, but the reduced homology of the $\jn _1 ^H S^0 \cong S^1$ is $\widetilde \bbZ$ since the $H=C_2$ action on the circle is given by reflection with respect to a line, not the antipodal action. 
\end{example}

Using Proposition \ref{pro:HomologyJoinInd}, we can give a topological proof for Bouc's tensor induction formula.

\begin{theorem}[Theorem 12.6.6 in \cite{Bouc-BisetBook}]\label{thm:TensorInd} Let $H$ be a $p$-group and $K$ be a subgroup of $H$, and let $X$ be a $K$-set. Then in $D(G)$,
$$\Ten _K ^H \Omega _X= \sum _{S,T\leq_H H,\ S\leq _H T} \mu _H (S,T)\, |\{ h\in T\backslash H / K \colon X^{T^h \cap K} \neq \emptyset \}|\, \Omega _{H/S} $$ where the $S\leq _H T$ means $S,T$ are taken from the poset of conjugacy classes of subgroups of $H$ and $\mu_H$ denotes the M\" obius function for this poset.
\end{theorem}

\begin{proof} Let $X$ be a $K$-set. Considering $X$ as a finite Moore $K$-space, we have $\Hn ([X])=\Omega _X$. By Proposition \ref{pro:HomologyJoinInd}, $\Ten _K ^H \Omega _X$ is equal to the equivalence class of the reduced homology module of $\jn _K ^H X$. By Proposition \ref{pro:Justifies} and Theorem \ref{thm:main}, this class is generated by relative syzygies so it belongs to $D^{\Omega} (H)$. This proves that tensor induction $\Ten _K ^H$ defines an operation $D^{\Omega} (K) \to D^{\Omega} (H) $. 

To verify the explicit formula given in the theorem, let us define $$F(X,T):=|\{ h\in T\backslash H / K \colon X^{T^h \cap K} \neq \emptyset \}|.$$ Since $$\Ten _K ^H \Omega _X = \Hn (\jn _K ^H X)=\Psi ( \Dim (\jn _K ^H X))= \Psi (\Jnd _K ^H \omega _X),$$ we need to show that  $$\Jnd _K ^H \omega _X=\sum _{S,T\leq_H H,\ S\leq _H T} \mu _H (S,T) F(X,T) \omega _{H/S}.$$
For every $T \leq H$, we have
\begin{equation}\begin{split} (\Jnd _K ^H \omega _X)(H/T) & = \omega _X (\Res ^H _K (H/T)) = \omega _X \Bigl (\sum _{h\in K\backslash H/T} K/(K\cap {}^h T ) \Bigr ) =F(X,T).
\end{split}
\end{equation}
Hence,
\begin{equation}
\begin{split} 
\Jnd_K ^H \omega _X & =\sum _{T \leq_H  H} (\Jnd _K ^H \omega _X )(H/T) e_T = \sum _{T \leq _H H } F(X,T) \sum _{S\leq _H T} \mu _H (S,T) \omega _{H/S}.\\
\end{split}
\end{equation}
gives the desired equality.
\end{proof}

We conclude the paper with the following result.

\begin{proposition} The commuting diagram in Theorem \ref{thm:mainSeq} is a diagram of biset functors. 
\end{proposition}

\begin{proof} We already proved in Proposition \ref{pro:DimCommutes} that $\Dim$ is a natural transformation of biset functors. By Proposition \ref{pro:HomologyJoinInd}, we know  that the homomorphism $\Hn$ commutes with induction. It is easy to check that $\Hn$ commutes with the restriction, inflation, deflation and isogation bisets (see \cite[Proposition 12.6.5]{Bouc-BisetBook}). So we can conclude that $\Hn$ is a natural transformation of biset functors. To prove that $\Psi$ is a natural transformation of biset functors, we need to show that for every $f \in C(K)$ and for every $(H,K)$-biset $U$, the equality $\Psi (\Jnd _U f)=T_U (\Psi f)$ holds. Since $C(K)$ is generated by $\omega _X$, it is enough to show this for $f=\omega _X$. We have
\begin{equation*}
\begin{split} 
\Psi (\Jnd _U \omega _X)&=\Psi (\Jnd _U (\Dim X))=\Psi (\Dim (\jn _U X)) = \Hn (\jn _U X) \\ &=T_U (\Hn (X)) =T_U (\Psi (\Dim X))=T_U (\Psi (\omega _X)).
\end{split}
\end{equation*}
Hence the proof is complete.
\end{proof}

\end{document}